\documentclass[a4paper,12pt]{article}
\usepackage[utf8]{inputenc}
\usepackage[T1]{fontenc}
\usepackage{times}
\usepackage[bitstream-charter]{mathdesign}
\usepackage[all]{xy}
\usepackage{enumerate}
\usepackage{graphicx}
\RequirePackage{amsopn} 
\RequirePackage{amsbsy} 
\RequirePackage{amsmath} 
\RequirePackage{relsize} 
\RequirePackage{latexsym} 
\RequirePackage{theorem}%
\RequirePackage{thc}%
\newcounter{mysubsection}[section]%
\newcounter{mysubsubsection}[mysubsection]%
\theorembodyfont{\slshape}%
\newtheorem{corollary}[mysubsection]{Corollary}%
\newtheorem{lemma}[mysubsection]{Lemma}%
\newtheorem{proposition}[mysubsection]{Proposition}%
\newtheorem{theorem}[mysubsection]{Theorem}%
\theorembodyfont{\upshape}%
\newtheorem{example}[mysubsection]{Example}%
\newtheorem{examples}[mysubsection]{Examples}%
\def\qed{{\unskip\nobreak\hfil\penalty50%
  \hskip2em\hbox{}\nobreak\hfil$\square$%
  \parfillskip=0pt\finalhyphendemerits=0\par}}
\newenvironment{proof}%
  {\par\addvspace{\medskipamount}%
    \upshape%
    {\slshape\scshape
    Proof\hskip\labelsep}}%
  {\qed%
    \addvspace{\medskipamount}}%
\newcommand\ideal[1]{\mathfrak{\lowercase{#1}}}
\newcommand\rMod{\textbf{Mod}-}
\newcommand\Ann{\textrm{Ann}}
\newcommand\Clt[3]{\mathrm{Cl}_{#1}^{#2}(#3)} %

\newcommand\lc{\textrm{lc}}
\newcommand\Max{\textrm{Max}}

\newcommand\sAm{\sigma_{A\setminus\ideal{m}}}
\newcommand\sAp{\sigma_{A\setminus\ideal{p}}}
\newcommand\sepa{\vspace*{-3mm}\setlength{\itemsep}{-2mm}}
\newcommand\serial[2]{{#1}\llbracket{#2}\rrbracket}
\newcommand\Spec{\textrm{Spec}}
\newcommand\Supp{\textrm{Supp}}
\newcommand\blfootnote[1]{%
  \begingroup
  \renewcommand\thefootnote{}\footnote{#1}%
  \addtocounter{footnote}{-1}%
  \endgroup}

\begin{document}
\title{An extension of $S$--noetherian rings and modules}
\author{P. Jara\\ {\normalsize Department of Algebra}\\ {\normalsize University of Granada}}
\date{}
\maketitle\blfootnote{pjara@ugr.es; \date{\today}}

\begin{abstract}
For any commutative ring $A$ we introduce a generalization of $S$--noetherian rings using a here\-ditary torsion theory $\sigma$ instead of a multiplicatively closed subset $S\subseteq{A}$. It is proved that if $A$ is a totally $\sigma$--noetherian ring, then $\sigma$ is of finite type, and that totally $\sigma$--noetherian is a local property.
\end{abstract}

\section*{Introduction}

In \cite{Hamann/Houston/Johnson:1988}, the authors study the problem of determining the structure of the polynomial ring $D[X]$, over an integral domain $D$ with field of fractions $K$, through the structure of the Euclidean domain $K[X]$. In particular, an ideal $\ideal{a}\subseteq{D[X]}$ is said to be \textbf{almost principal} whenever there exist a polynomial $F\in\ideal{a}$, of positive degree, and an element $0\neq{s}\in{D}$ such that $\ideal{a}s\subseteq{FD[X]}\subseteq\ideal{a}$. The integral domain $D$ is an \textbf{almost principal domain} whenever every ideal $\ideal{a}\subseteq{D[X]}$, which extends properly to $K[X]$, is almost principal.
Noetherian and integrally closed domains are examples of almost principal domains.

Later, in \cite{Anderson/Dumitrescu:2002}, the authors extend this notion to non--necessarily integral domains in defining, for a given multiplicatively closed subset $S\subseteq{A}$ of a ring $A$, an ideal $\ideal{a}\subseteq{A}$ to be \textbf{$S$--finite} if there exist a finitely generated ideal $\ideal{a}'\subseteq\ideal{a}$ and an element $s\in{S}$ such that $\ideal{a}s\subseteq\ideal{a}'$, and define a ring $A$ to be \textbf{$S$--noetherian} whenever every ideal $\ideal{a}\subseteq{A}$ is $S$--finite.
Many authors have worked on $S$--noetherian rings and related notions, and shown relevant results about its structure. See for instance
\cite{Eljeri:2018,
Hamed:2018,
Lim:2015,
Lim/Oh:2014,
Sevim/Tekir/Koc:2020,
Zhongkui:2007}.

The main aim of this paper is to give a new approach to $S$--noetherian rings and modules, and its applications, using the more abstract notion of hereditary torsion theory. From this new point of view, several results appear more evident and appear inlaid in a more general theory, which clarifies the original approach.

The background we use will be the hereditary torsion theories on a commutative (and unitary) ring $A$, see \cite{STENSTROM:1975, GOLAN:1986}, and we denote by $\rMod{A}$ the category of $A$--modules. Thus, a hereditary torsion theory $\sigma$ in $\rMod{A}$ is given by one of the following objects:
\begin{enumerate}[(1)]\sepa
\item
a \textbf{torsion class} $\mathcal{T}_\sigma$, a class of modules which is closed under submodules, homomorphic images, direct sums and group extensions,
\item
a \textbf{torsionfree class} $\mathcal{F}_\sigma$, a class of modules which is closed under submodules, essential extensions, direct products and group extensions,
\item
a \textbf{Gabriel filter} of ideals $\mathcal{L}(\sigma)$, a non--empty filter of ideals satisfying that every $\ideal{b}\subseteq{A}$, for which there exists an ideal $\ideal{a}\in\mathcal{L}(\sigma)$ such that $(\ideal{b}:a)\in\mathcal{L}(\sigma)$, for every $a\in\ideal{a}$, belongs to $\mathcal{L}(\sigma)$.
\item
a \textbf{left exact kernel functor} $\sigma:\rMod{A}\longrightarrow\rMod{A}$.
\end{enumerate}
The relationships between these notions are the following. If $\sigma$ is the left exact kernel functor, then
\[
\begin{array}{ll}
\mathcal{T}_\sigma=\{M\in\rMod{A}\mid\;\sigma{M}=M\},\\
\mathcal{F}_\sigma=\{M\in\rMod{A}\mid\;\sigma{M}=0\},\\
\mathcal{L}(\sigma)=\{\ideal{a}\subseteq{A}\mid\;A/\ideal{a}\in\mathcal{T}_\sigma\}.\\
\end{array}
\]
If $\mathcal{L}$ is the Gabriel filter of a hereditary torsion theory $\sigma$, and $\mathcal{T}$ is the torsion class, for any $A$--module $M$ we have:
\[
\sigma{M}
=\{m\in{M}\mid\;(0:m)\in\mathcal{L}\}
=\sum\{N\subseteq{M}\mid\;N\in\mathcal{T}\}.\\
\]

\begin{example}
\begin{enumerate}[(1)]\sepa
\item
Let $\Sigma\subseteq{A}$ be a multiplicatively closed subset, there exists a hereditary torsion theory, $\sigma_\Sigma$, defined by
\[
\mathcal{L}(\sigma_\Sigma)=\{\ideal{a}\subseteq{A}\mid\;\ideal{a}\cap\Sigma\neq\varnothing\}.
\]
Observe that $\sigma_\Sigma$ has a filter basis constituted by principal ideals. A hereditary torsion theory $\sigma$ such that $\mathcal{L}(\sigma)$ has a filter basis of principal ideals is called a \textbf{principal hereditary torsion theory}.
We can show there is a correspondence between principal hereditary torsion theories in $\rMod{A}$, and saturated multiplicatively closed subsets in $A$.
\item
For any ring $A$ the set $\mathcal{L}=\{\ideal{a}\subseteq{A}\mid\;\Ann(\ideal{a})=0\}$ is a Gabriel filter, it defines a hereditary torsion theory which we represent by $\lambda$.
\end{enumerate}
\end{example}

\medskip

This paper is organized in sections. In the first one we introduce totally $\sigma$--noetherian rings and modules and show that necessarily the hereditary torsion theory $\sigma$ is of finite type whenever the ring $A$ is totally $\sigma$--noetherian. In section two we study how prime ideals and prime submodules appears naturally when studying totally $\sigma$--finitely generated modules and obtain a relative version of Cohen's theorem. The natural notion of maximal condition in relation with totally $\sigma$--noetherian modules is studied in section three. Section four is devoted to study some extensions of totally $\sigma$--noetherian rings. In particular, we find that it is necessary to impose some extra conditions to $\sigma$ in order to assure that $A[X]$ is totally $\sigma$--noetherian whenever $A$ is. The local behaviour of totally $\sigma$--noetherian modules is studied in section four in which we can reduce to consider only prime ideals in $\mathcal{K}(\sigma)$. In the last section we study the particular case of principal ideal rings.

Through this paper we try to study $\sigma$--noetherian rings and modules and, in a parallel way, totally $\sigma$-noetherian rings and modules. The first one ($\sigma$--noetherian) has a categorical behaviour, but not the second one (totally $\sigma$--noetherian). For that  reason, the study of the last one is more difficult and it is not in the literature. Otherwise, it produces results  as Proposition~\eqref{pr:20191030} which shows that to be totally $\sigma$--noetherian, as opposed to $\sigma$-noetherian, is a local property.

\section{Totally $\sigma$--noetherian rings and modules}

For any $\sigma$--torsion finitely generated $A$--module $M$, say $M=m_1A+\cdots+m_tA$, since $(0:m_i)\in\mathcal{L}(\sigma)$, for any $i=1,\ldots,t$, if $\ideal{h}:=\cap_{i=1}^t(0:m_i)\in\mathcal{L}(\sigma)$, they satisfy $M\ideal{h}=0$. In general, this result does not hold for $\sigma$--torsion non--finitely generated $A$--modules. Therefore, we shall define an $A$--module $M$ to be \textbf{totally $\sigma$--torsion} whenever there exists $\ideal{h}\in\mathcal{L}(\sigma)$ such that $M\ideal{h}=0$. This notion of totally torsion appears, for instance, in \cite[page 462]{Jategaonkar:1971}.

For any ideal $\ideal{a}\subseteq{A}$ we have two different notions of finitely generated ideals relative to $\sigma$:
\begin{enumerate}[(1)]\sepa
\item
$\ideal{a}\subseteq{A}$ is \textbf{$\sigma$--finitely generated} whenever there exists a finitely generated ideal $\ideal{a}'\subseteq\ideal{a}$ such that $\ideal{a}/\ideal{a}'$ is $\sigma$--torsion.
\item
$\ideal{a}\subseteq{A}$ is \textbf{totally $\sigma$--finitely generated} whenever there exists a finitely generated ideal $\ideal{a}'\subseteq\ideal{a}$ such that $\ideal{a}/\ideal{a}'$ is totally $\sigma$--torsion.
\end{enumerate}

In the same way, for any ring $A$ we have two different notions of noetherian ring relative to $\sigma$:
\begin{enumerate}[(1)]\sepa
\item
$A$ is \textbf{$\sigma$--noetherian} if every ideal is $\sigma$--finitely generated.
\item
$A$ is \textbf{totally $\sigma$--noetherian} whenever every ideal is totally $\sigma$--finitely generated.
\end{enumerate}

\medskip
\begin{example}
\begin{enumerate}[(1)]\sepa
\item
Every finitely generated ideal is totally $\sigma$--finitely generated and every totally $\sigma$--finitely generated ideal is $\sigma$--finitely generated.
\item
Let $S\subseteq{A}$ be a multiplicatively closed subset, an ideal $\ideal{a}\subseteq{A}$ is $S$--finite if, and only if, it is totally $\sigma_S$--finitely generated. The ring $A$ is $S$--noetherian if, and only if, $A$ is totally $\sigma_S$--noetherian
\end{enumerate}
\end{example}

These two notions of torsion, and the notions derived from them, are completely different in its behaviour and its categorical properties. For instance, due to the definition, for any $A$--module $M$ there exists a maximum submodule belonging to $\mathcal{T}_\sigma$, the submodule: $\sigma{M}$, and it satisfies $M/\sigma{M}\in\mathcal{F}(\sigma)$. In the totally $\sigma$--torsion case we can not assure the existence of a maximal totally $\sigma$--torsion submodule. The existence of a maximum $\sigma$--torsion submodule allows us to build new concepts relative to $\sigma$ as lattices, closure operators and localizations; concepts that we have not in the totally $\sigma$--torsion case. Nevertheless, the totally $\sigma$--torsion case allows us to study arithmetic properties of rings and modules which are hidden with that use of $\sigma$--torsion, and these properties are those which we are interested in studying.

As we pointed out before, the $\sigma$--torsion allows, for any $A$--module $M$, to define a lattice \[
C(M,\sigma)=\{N\subseteq{M}\mid\;M/N\in\mathcal{F}_\sigma\},
\]
and a closure operator $\Clt{\sigma}{M}{-}:\mathcal{L}(M)\longrightarrow{C(M,\sigma)}\subseteq\mathcal{L}(M)$, from $\mathcal{L}(M)$, the lattice of all submodules of $M$, defined by the equation $\Clt{\sigma}{M}{N}/N=\sigma(M/N)$. The elements in $C(M,\sigma)$ are called the \textbf{$\sigma$--closed} submodules of $M$, and the lattice operations in $C(M,\sigma)$, for any $N_1,N_2\in{C(M,\sigma)}$, are defined by
\[
\begin{array}{ll}
N_1\wedge{N_2}=N_1\cap{N_2},\\
N_1\vee{N_2}=\Clt{\sigma}{M}{N_1+N_2}.
\end{array}
\]
Dually, the submodules $N\subseteq{M}$ such that $M/N\in\mathcal{T}_\sigma$ are called \textbf{$\sigma$--dense} submodules. The set of all $\sigma$--dense submodules of $M$ is represented by $\mathcal{L}(M,\sigma)$.

\medskip
In the following, we assume $A$ is a ring, $\rMod{A}$ is the category of $A$--modules and $\sigma$ is a hereditary torsion theory on $\rMod{A}$. Modules are represented by Latin letters: $M,N,N_1,\ldots$, and ideals by Gothics letters: $\ideal{a},\ideal{b},\ideal{b}_1,\ldots$ Different hereditary torsion theories will be represented by Greek letters: $\sigma,\tau,\sigma_1,\ldots$, and induced hereditary torsion theories by adorned Greek letters: $\sigma',\overline{\tau},\ldots$

The notions of (totally) $\sigma$--finitely generated and (totally) $\sigma$--noetherian can be extended to $A$--modules in an easy way. Properties on the behaviour of totally $\sigma$--finitely generated and $\sigma$--noethe\-rian modules are collected in the following result.

\medskip
\begin{proposition}
\begin{enumerate}[(1)]\sepa
\item
Every homomorphic image of a totally $\sigma$--finitely generated $A$--module also is .
\item
For every submodule $N\subseteq{M}$, we have: $M$ is totally $\sigma$--noetherian if, and only if, $N$ and $M/N$ are totally $\sigma$--noetherian.
\item
Finite direct sums of totally $\sigma$--noetherian modules also are.
\end{enumerate}
\end{proposition}

The first restriction we have found in studying totally $\sigma$--noetherian rings is that the hereditary torsion theory $\sigma$ must be of an special type: it is a \textbf{finite type} hereditary torsion theory. This is an extension of principal hereditary torsion theories, and means that $\mathcal{L}(\sigma)$ has a filter basis constituted by finitely generated ideals.

\begin{proposition}
If $A$ is a totally $\sigma$--noetherian ring then $\sigma$ is of finite type.
\end{proposition}
\begin{proof}
For any $\ideal{a}\in\mathcal{L}(\sigma)$ there exists $\ideal{a}'\subseteq\ideal{a}$, finitely generated, and $\ideal{h}\in\mathcal{L}(\sigma)$ such that $\ideal{a}\ideal{h}\subseteq\ideal{a}'$. Since $\mathcal{L}(\sigma)$ is closed under product of ideals, we have $\ideal{h}'\in\mathcal{L}(\sigma)$.
\end{proof}

Directly from the definition we have that every totally $\sigma$--torsion module is totally $\sigma$--noetherian, and an $A$--module $M$ is totally $\sigma$--noetherian if, and only if it contains a finitely generated submodule $N\subseteq{M}$ such that $M/N$ is totally $\sigma$--torsion. Our aim is to explore more conditions equivalent to totally $\sigma$--noetherian. In this sense, since every totally $\sigma$--noetherian module is $\sigma$--noetherian, the question is what properties are necessary to add to $\sigma$--noetherianness to get totally $\sigma$--noetherian.

The next proposition is based on \cite[Proposition 2]{Anderson/Dumitrescu:2002}.

\begin{proposition}\label{pr:140706}
Let $\sigma$ be a finite type hereditary torsion theory, and $M$ be an $A$--module, the following statements are equivalent:
\begin{enumerate}[(a)]\sepa
\item
$M$ is totally $\sigma$--noetherian.
\item
$M$ is $\sigma$--noetherian, and for every $H\subseteq{M}$, finitely generated, there exists $\ideal{h}\in\mathcal{L}(\sigma)$, finitely generated, such that $\Clt{\sigma}{M}{H}=(H:\ideal{h})$.
\end{enumerate}
\end{proposition}
\begin{proof}
(a) $\Rightarrow$ (b). %
If $\Clt{\sigma}{M}{N}$ is totally $\sigma$--finitely generated, there exist $H\subseteq\Clt{\sigma}{M}{N}$, finitely generated, and $\ideal{h}\in\mathcal{L}(\sigma)$, finitely generated, such that $\Clt{\sigma}{M}{N}\ideal{h}\subseteq{H}\subseteq\Clt{\sigma}{M}{N}$. There exists $\ideal{b}\in\mathcal{L}(\sigma)$, finitely generated, such that $H\ideal{b}\subseteq{N}$. Therefore,
\[
N\ideal{h}\ideal{b}
\subseteq\Clt{\sigma}{M}{N}\ideal{h}\ideal{b}
\subseteq{H\ideal{h}\ideal{b}}
\subseteq{H\ideal{b}}
\subseteq{N}.
\]
In particular, $\Clt{\sigma}{M}{N}\subseteq(N:\ideal{h}\ideal{b})$. On the other hand, $(N:\ideal{h}\ideal{b})\ideal{h}\ideal{b}\subseteq{N}$, hence $\Clt{\sigma}{M}{N}=(N:\ideal{h}\ideal{b})$.
\par
(b) $\Rightarrow$ (a). %
Let $N\subseteq{M}$, there is $H\subseteq{N}$, finitely generated, such that $\Clt{\sigma}{M}{N}=\Clt{\sigma}{M}{H}$, and there exists $\ideal{h}\in\mathcal{L}(\sigma)$, finitely generated such that $\Clt{\sigma}{M}{H}=(H:\ideal{h})$. Therefore we have:
\[
N\ideal{h}\subseteq\Clt{\sigma}{M}{N}\ideal{h}=\Clt{\sigma}{M}{H}\ideal{h}\subseteq{H}\subseteq{N}.
\]
\end{proof}

We know how to induce hereditary torsion theories through a ring map. Here we shall study the particular case of a ring map $f:A\longrightarrow{B}$ such that every ideal of $B$ is extended of an ideal of $A$, i.e., for any ideal $\ideal{b}\subseteq{B}$ there exists an ideal $\ideal{a}\subseteq{A}$ such that $f(\ideal{a})B=\ideal{b}$

Let $\sigma$ be a hereditary torsion theory in $\rMod{A}$, then $f(\sigma)$ is a hereditary torsion theory in $\rMod{B}$ and its Gabriel filter is
\[
\mathcal{L}(f(\sigma))=\{\ideal{b}\subseteq{B}\mid\;f^{-1}(\ideal{b})\in\mathcal{L}(\sigma)\}.
\]
It is clear that $f(\sigma)$ is of finite type whenever $\sigma$ is.

In this situation we have:

\begin{lemma}
Let $\sigma$ be a finite type hereditary torsion theory, $f:A\longrightarrow{B}$ be a ring map such that every ideal of $B$ is an extended ideal. In this case $\mathcal{L}(f(\sigma))=\{f(\ideal{a})B\mid\;\ideal{a}\in\mathcal{L}(\sigma)\}$.
\par
If $A$ is totally $\sigma$--noetherian, then $B$ is totally $f(\sigma)$--noetherian.
\end{lemma}
\begin{proof}
Let $\ideal{b}\subseteq{B}$, there exists $\ideal{a}\subseteq{A}$ such that $\ideal{b}=f(\ideal{a})B$. There exists $\ideal{h}\in\mathcal{L}(\sigma)$, finitely generated, such that $\ideal{a}\ideal{h}\subseteq\ideal{a}'\subseteq\ideal{a}$, for some finitely generated ideal $\ideal{a}'\subseteq\ideal{a}$. Therefore, $\ideal{b}f(\ideal{h})B=f(\ideal{a})f(\ideal{h})B\subseteq{f(\ideal{a}')}B\subseteq\ideal{b}$.
\end{proof}

Examples of this situation are the following:
\begin{enumerate}[(1)]\sepa
\item
$B$ is the quotient of a ring $A$ by an ideal $\ideal{a}$, i.e., $p:A\longrightarrow{A/\ideal{a}}$.
\item
$B$ is the localized ring of $A$ at a multiplicatively closed subset $\Sigma\subseteq{A}$, i.e., $q:A\longrightarrow\Sigma^{-1}A$.
\end{enumerate}

\section{Prime ideals}

If $\sigma$ is a hereditary torsion theory in $\rMod{A}$, it is well known that for any prime ideal $\ideal{p}\subseteq{A}$ we have either $\ideal{p}\in{C(A,\sigma)}$ or $\ideal{p}\in\mathcal{L}(\sigma)$, i.e., either $A/\ideal{p}$ is $\sigma$--torsionfree or $A/\ideal{p}$ is $\sigma$--torsion. In consequence, $\sigma$ produces a partition of $\Spec(A)$ in two sets: $\Spec(A)=\mathcal{K}(\sigma)\cup\mathcal{Z}(\sigma)$, with $\mathcal{K}(\sigma)\subseteq{C(A,\sigma)}$, and $\mathcal{Z}(\sigma)\subseteq\mathcal{L}(\sigma)$. In addition, for every $\ideal{p}\in\mathcal{K}(\sigma)$ we have $\sigma\leq\sAp$, and $\sigma=\wedge\{\sAp\mid\;\ideal{p}\in\mathcal{K}(\sigma)\}$, whenever $\sigma$ is of finite type.

\begin{proposition}
Let $\sigma$ be a finite type hereditary torsion theory in $\rMod{A}$, and let $M$ be a totally $\sigma$--finitely generated module. If $N\subseteq{M}$ is maximal among all non--totally $\sigma$--finitely generated submodules of $M$, then $(N:M)$ is a prime ideal.
\end{proposition}
\begin{proof}
Let $\ideal{p}=(N:M)$; if $\ideal{p}$ is not prime, there exist $a,b\in{A\setminus\ideal{p}}$ such that $ab\in\ideal{p}$. As a consequence, since $N\subsetneqq{N+Mb}$, then $N+Mb$ is totally $\sigma$--finitely generated, and $N,Ma\subseteq(N:b)$, hence $N\subsetneqq{N+Ma}\subseteq(N:b)$, then $(N:b)$ is totally $\sigma$--finitely generated.
\par
Let $\ideal{h}=\langle{h_1,\ldots,h_s}\rangle\in\mathcal{L}(\sigma)$ such that $(N+Mb)\ideal{h}\subseteq{F}=\langle{f_1,\ldots,f_t}\rangle\subseteq{N+Mb}$, and $(N:b)\ideal{h}\subseteq{G}=\langle{g_1,\ldots,g_r}\rangle\subseteq(N:b)$.
Let $f_j=n_j+m_jb$ for any $j=1,\ldots,t$, where $n_j\in{N}$ and $m_j\in{M}$.
\par
For any $n\in{N}$ and any $h_i\in\ideal{h}$ there exists an $A$--linear combination $nh_i=\sum_j(n_j+m_jb)c_{i,j}$, with $c_{i,j}\in{A}$. Let $x_i=\sum_jm_jc_{i,j}$; therefore $x_ib=\sum_jm_jc_{i,j}b=na_i+\sum_jn_jc_{i,j}\in{N}$, hence $x_i\in(N:b)$.
\par
Let us represent now the generators of $\ideal{h}$ by $h'_l$. For any $h'_l$ we have $x_ih'_l\in{G}$, and there exists an $A$--linear combination $x_ih'_l=\sum_kg_kd_{l,i,k}$, with $d_{l,i,k}\in{A}$. Thus we have:
\[
\begin{array}{ll}
nh_ih'_l
&=(nh_i)h'_l
 =(\sum_jn_jc_{i,j}+\sum_jm_jc_{i,j}b)h'_l
 =\sum_jn_jc_{i,j}h'_l+\sum_jm_jc_{i,j}bh'_l\\
&=\sum_j(n_jh'_l)c_{i,j}+x_ibh'_l
 =\sum_j(n_jh'_l)c_{i,j}+(\sum_kg_kd_{l,i,k})b\\
&=\sum_j(n_jh'_l)c_{i,j}+\sum_k(g_kb)d_{l,i,k}
 \in\langle{n_jh'_l,g_kb\mid\;l,j=1,\ldots,s;\;k=1,\ldots,r}\rangle\subseteq{N}.
\end{array}
\]
In particular,
$N\ideal{h}\ideal{h}\subseteq
\langle{n_jh'_l,g_kb\mid\;l,j=1,\ldots,s;\;k=1,\ldots,r}\rangle\subseteq{N}$.
This means that $N$ is totally $\sigma$--finitely generated, which is a contradiction.
\end{proof}

If $N\subseteq{M}$ is maximal among the non totally $\sigma$--finitely generated submodules, is it a prime submodule? We know that it holds in the case of finitely generated modules. Let us prove it now for totally $\sigma$--finitely generated modules.

\begin{proposition}
Let $\sigma$ be a finite type hereditary torsion theory, and $M$ be a totally $\sigma$--finitely generated $A$--module. Any $N\subseteq{M}$, maximal among the submodules of $M$ which are not totally $\sigma$--finitely generated, is a prime submodule.
\end{proposition}
\begin{proof}
Let $N\subseteq{M}$ such a maximal submodule. If $N\subseteq{M}$ is not prime, there exist $m\in{M\setminus{N}}$ and $a\in{A}\setminus(N:M)$ such that $ma\in{N}$.
\par
Since $a\notin(N:M)$, then $Ma\nsubseteq{N}$, and $N\subsetneqq{N+Ma}$ is totally $\sigma$--finitely
generated. On the other hand, since $m\in(N:a)\setminus{N}$, then $N\subsetneqq(N:a)$ is totally $\sigma$--finitely generated. Therefore, there exist a finitely generated submodules $F=\langle{f_1,\ldots,f_r}\rangle\subseteq{N+Ma}$ and $G=\langle{g_1,\ldots,g_s}\rangle\subseteq(N:a)$, and $\ideal{h}=\langle{h_1,\ldots,h_t}\rangle\in\mathcal{L}(\sigma)$, finitely generated, such that $(N+Ma)\ideal{h}\subseteq{F}\subseteq{N+Ma}$ and $(N:a)\ideal{h}\subseteq{G}\subseteq(N:a)$. Say $f_i=n_i+m_ia$ for $i=1,\ldots,r$, $n_i\in{N}$ and $m_i\in{M}$.
\par
For any $n\in{N}$ and $h_i\in\{h_1,\ldots,h_t\}$, since $nh_i\in{F}$, there exists an $A$--linear combination $nh_i=\sum_lf_lc_{i,l}=\sum_ln_lc_{i,l}+\sum_lm_hc_{i,l}a$, and $(\sum_lm_lc_{i,l})a=n-\sum_ln_lc_{i,l}\in{N}$. In consequence $\sum_lm_lc_{i,l}\in(N:a)$.
\par
For any $h_j\in\{h_1,\ldots,h_t\}$ we have $(\sum_lm_lc_{i,l})h_j\in{G}$, and there exists an $A$--linear combination
$$
\left(\sum_lm_lc_{i,l}\right)h_j=\sum_kg_kd_{i,j,k}\textrm{, with }d_{i,j,k}\in{A}.
$$
Therefore, we have
\begin{multline*}
nh_ih_j
=\left(\sum_ln_lc_{i,l}+\sum_lm_lc_{i,l}a\right)h_j
=\sum_ln_lc_{i,l}h_j+\sum_lm_lc_{i,l}h_ja\\
=\sum_ln_lc_{i,l}h_j+\sum_kg_kd_{i,j,k}a,
\end{multline*}
which means that $n\ideal{h}\ideal{h}$ is contained in the submodule of $N$ generated by $\{n_1,\ldots,n_r\}$ and $\{g_1a,\ldots,g_sa\}$.
\end{proof}

The next one is a properly result of finite type hereditary torsion theories.

\begin{lemma}
Let $\sigma$ be a finite type hereditary torsion theory. For every totally $\sigma$--finitely generated $A$--module $M$, and any $L\in{C(M,\sigma)}$, $L\subsetneqq{M}$, there exists a maximal element $N\in{C(M,\sigma)}$ such that $L\subseteq{N}$. In addition, if  $\Gamma=\{N\subseteq{M}\mid\;L\subseteq{N}\in{C(M,\sigma)},\,N\neq{M}\}$, every maximal element in $\Gamma$ is a prime submodule of $M$.
\end{lemma}
\begin{proof}
Let $\{N_i\}_i$ be a chain in $\Gamma$, and we define $N=\cup_{i\in{I}}N_i$.
If $\Clt{\sigma}{M}{N}\neq{M}$, then $\Clt{\sigma}{M}{N}$ is an upper bound of the chain in $\Gamma$.
If $\Clt{\sigma}{M}{N}=M$, since there exist $m_1,\ldots,m_t\in{M}$ and $\ideal{h}\in\mathcal{L}(\sigma)$, finitely generated, such that $M\ideal{h}\subseteq(m_1,\ldots,m_t)A\subseteq{M}$, then there exists $\ideal{b}\in\mathcal{L}(\sigma)$, finitely generated, such that $(m_1,\ldots,m_t)\ideal{b}\subseteq\cup_iN_i=N$; therefore, there exists an index $i$ such that $(m_1,\ldots,m_t)\ideal{b}\subseteq{N_i}$. In consequence, $N_i=\Clt{\sigma}{M}{N_i}=M$, which is a contradiction.
\end{proof}

The following result is based in \cite[Theorem~1]{Jothilingam:2000}, see also \cite[Proposition~4]{Anderson/Dumitrescu:2002}.

\begin{proposition}\label{pr:140707}
Let $\sigma$ be a finite type hereditary torsion theory in $\rMod{A}$, and let $M$ be a totally $\sigma$--finitely generated module. The following statements are equivalent:
\begin{enumerate}[(a)]\sepa
\item
$M$ is totally $\sigma$--noetherian.
\item
For every prime ideal $\ideal{p}\in\mathcal{K}(\sigma)$ the submodule $M\ideal{p}\subseteq{M}$ is totally $\sigma$--finitely generated.
\end{enumerate}
\end{proposition}
\begin{proof}
Clearly, if $\ideal{p}\in\mathcal{Z}(\sigma)$, then $M\ideal{p}\subseteq{M}$ is $\sigma$--dense, hence totally $\sigma$--finitely generated, because $M$ is.
\par
(a) $\Rightarrow$ (b). %
It is obvious.
\par
(b) $\Rightarrow$ (a). %
Since $M$ is totally $\sigma$--finitely generated, there exists $\ideal{h}\in\mathcal{L}(\sigma)$, finitely generated, such that $\ideal{h}M\subseteq(a_1,\ldots,a_t)\subseteq{M}$. If $M$ is not totally $\sigma$--noetherian, $$
\Gamma=\{N\subseteq{M}\mid\;N\mbox{ is not totally $\sigma$--finitely generated}\}
$$
is not empty. Any chain in $\Gamma$ has a upper bound in $\Gamma$, hence, by Zorn's lemma, there exists $N\in\Gamma$ maximal. The ideal $\ideal{p}=(N:M)$ is prime; if $\ideal{p}\in\mathcal{Z}(\sigma)$, then $N$ is totally $\sigma$--finitely generated as $M$ is. Therefore, $\ideal{p}\in\mathcal{K}(\sigma)$, and, by the hypothesis, $M\ideal{p}$ is totally $\sigma$--finitely generated, there exists $\ideal{h}'\in\mathcal{L}(\sigma)$ such that $M\ideal{p}\ideal{h}'\subseteq(b_1,\ldots,b_s)\subseteq{M}\ideal{p}$.
\par
Since $\ideal{p}=(N:M)\subseteq(N:(a_1,\ldots,a_t))\subseteq(N:M\ideal{h})=(\ideal{p}:\ideal{h})=\ideal{p}$, as $\ideal{p}$ is prime, the $\ideal{p}=(N:a_1)\cap\ldots\cap(N:a_t)$, hence there exists an index $i$ such that $\ideal{p}=(N:a_i)$, and $a_i\notin{N}$. Therefore $N+A a_i\supsetneqq{N}$, and $N+A a_i$ is totally $\sigma$--finitely generated. There exists $\ideal{h}''\in\mathcal{L}(\sigma)$ such that $N\ideal{h}''\subseteq(n_1+x_1 a_i,\ldots,n_r+x_r a_i)\subseteq{N+Aa_i}$, where $x_1,\ldots,x_r\in{A}$, and $N\ideal{h}''\subseteq(n_1,\ldots,n_r)+a_i\ideal{p}$.
Then $N\ideal{h}''\ideal{h}'\subseteq(n_1,\ldots,n_r)\ideal{h}'+(b_1,\ldots,b_s)\subseteq{N+M\ideal{p}}\subseteq{N}$, which is a contradiction.
\end{proof}

As a direct consequence we have:

\begin{corollary}[Cohen's theorem]
Let $\sigma$ be a finite type hereditary torsion theory in $\rMod{A}$, the following statements are
equivalent:
\begin{enumerate}[(a)]\sepa
\item
$A$ is totally $\sigma$--noetherian.
\item
Every prime ideal in $\mathcal{K}(\sigma)$ is totally $\sigma$--finitely generated.
\end{enumerate}
\end{corollary}
\begin{proof}
It is a direct consequence of Proposition~\eqref{pr:140707}.
\end{proof}

When we particularize to the hereditary torsion theory $\sigma=0$, i.e, when $\mathcal{L}(\sigma)=\{A\}$, we have that $A$ is a noetherian ring if, and only if, every prime ideal is finitely generated, which is Cohen's Theorem. Otherwise, if $\sigma=\sigma_S$, for some multiplicatively closed subset $S\subseteq{A}$, then $A$ is $S$--noetherian if, and only if, every prime ideal, in $\mathcal{K}(\sigma_S)$, is totally $S$--finite, see \cite{Anderson/Dumitrescu:2002}.

\section{Maximal conditions}

Let $M$ be an $A$--module, an increasing chain of submodules $N_1\subseteq{N_2}\subseteq\cdots$ is \textbf{totally $\sigma$--stable} whenever there exist an index $m$ and $\ideal{h}\in\mathcal{L}(\sigma)$ such that $N_m\ideal{h}\subseteq{N_s}$ for every $s\geq{m}$.

\begin{proposition}
Let $\sigma$ be a hereditary torsion theory and $M$ be an $A$--module. The following statements are equivalent:
\begin{enumerate}[(a)]\sepa
\item
$M$ is totally $\sigma$--noetherian.
\item
Every increasing chain $N_1\subseteq{N_2}\subseteq\cdots$ is totally $\sigma$--stable.
\end{enumerate}
\end{proposition}
\begin{proof}
(a) $\Rightarrow$ (b). %
Let $N_1\subseteq{N_2}\subseteq\cdots$ be an increasing chain of submodules of $M$, and define $N=\cup_{s\geq1}N_s$. By the hypothesis, there exist $\ideal{h}\subseteq\mathcal{L}(\sigma)$ and $n_1,\ldots,n_t\in{N}$ such that $N\ideal{h}\subseteq(n_1,\ldots,n_t)\subseteq{N}$. Therefore, there exists an index $m$ such that $n_1,\ldots,n_t\in{N_m}$, and we have $N\ideal{h}\subseteq{N_m}$. In particular, $N_s\ideal{h}\subseteq{N_m}$ for every $s\geq{m}$.
\par
(b) $\Rightarrow$ (a). %
Let $N\subseteq{M}$, for any increasing chain $\{N_s\}_{s\in\mathbb{N}}$ of totally $\sigma$--finitely generated submodules of $N$ there exist an index $m$ and $\ideal{h}\in\mathcal{L}(\sigma)$ such that $N_s\ideal{h}\subseteq{N_m}$, for any $s\geq{m}$. In particular, $N\ideal{h}\subseteq{N_m}$, and $N$ is totally $\sigma$--finitely generated. This means that the family of all totally $\sigma$--finitely generated submodules of $N$ has maximal elements. Let $H\subseteq{N}$ one of these maximal elements. If $H\subsetneqq{N}$, there exists $n\in{N\setminus{H}}$, hence $H+nA\subseteq{N}$ is totally finitely generated, which is a contradiction.
\end{proof}

Let $M$ be an $A$--module, we have the following definitions:
\begin{enumerate}[(1)]\sepa
\item
Let $\mathcal{N}\subseteq\mathcal{L}(M)$ be a family of submodules of $M$. An element $N\in\mathcal{N}$ is \textbf{$\sigma$--maximal} if there exists $\ideal{h}\in\mathcal{L}(\sigma)$ such that for every $H\in\mathcal{N}$ satisfying $N\subseteq{H}$ we have $H\ideal{h}\subseteq{N}$.
\item
The $A$--module $M$ satisfies the \textbf{$\sigma$-MAX condition} if every nonempty family of submodules of $M$ has $\sigma$--maximal elements.
\item
A family $\mathcal{N}$ of submodules of $M$ is \textbf{$\sigma$--upper closed} if for every submodule $H\subseteq{M}$ such that there exist $N\in\mathcal{N}$ and $\ideal{h}\in\mathcal{L}(\sigma)$ satisfying $H$ and $H\ideal{h}\subseteq{N}$, or equivalently $H\subseteq(N:\ideal{h})$, we have $H\in\mathcal{N}$.
\end{enumerate}

\begin{proposition}
Let $M$ be an $A$--module, the following statements are equivalent:
\begin{enumerate}[(a)]\sepa
\item
$M$ is totally $\sigma$--noetherian.
\item
Every nonempty $\sigma$--upper closed family of submodules of $M$ has maximal elements.
\item
Every nonempty family of submodules of $M$ has $\sigma$--maximal elements.
\end{enumerate}
\end{proposition}
\begin{proof}
(a) $\Rightarrow$ (b). %
Let $\mathcal{N}$ be a nonempty $\sigma$--upper closed family of submodules of $M$. For any increasing chain $N_1\subseteq{N_2}\subseteq\cdots$ in $\mathcal{N}$ we define $N=\cup_{s\geq1}N_s$. By the hypothesis there exist an index $m$ and $\ideal{h}\in\mathcal{L}(\sigma)$ such that $N_s\ideal{h}\subseteq{N_m}$, for every $s\geq{m}$. Hence $N\ideal{h}\subseteq{N_m}$, and $N\in\mathcal{N}$. In consequence, by Zorn's lemma, $\mathcal{N}$ contains maximal elements.
\par
(b) $\Rightarrow$ (c). %
Let $\mathcal{N}$ be a nonempty family of submodules of $M$. We define a new family
$$
\overline{\mathcal{N}}=\{H\subseteq{M}\mid\;\textrm{there exist }N\in\mathcal{N}\textrm{ and }\ideal{h}\in\mathcal{L}(\sigma)\textrm{ such that }H\ideal{h}\subseteq{N}\},
$$
the \textbf{$\sigma$--upper closure} of $\mathcal{N}$.
We claim $\overline{\mathcal{N}}$ is $\sigma$--upper closed. Indeed, if $L\subseteq{M}$, $H\in\overline{\mathcal{N}}$ and $\ideal{h}\in\mathcal{L}(\sigma)$ satisfy $L\subseteq(H:\ideal{h})$, by the hypothesis there exist $N\in\mathcal{N}$ and $\ideal{h}'\in\mathcal{L}(\sigma)$ such that $H\subseteq(N:\ideal{h}')$, hence we have
$L\subseteq(H:\ideal{h})
\subseteq((N:\ideal{h}'):\ideal{h})
=(N:\ideal{h}'\ideal{h})$,
and $L\in\overline{\mathcal{N}}$.
\par
By the hypothesis, there exists a maximal element, say $H$, in $\overline{\mathcal{N}}$, and there exist $N\in\mathcal{N}$ and $\ideal{h}\in\mathcal{L}(\sigma)$ such that $N\subseteq{H}\subseteq(N:\ideal{h})$. Since $(N:\ideal{h})\in\overline{\mathcal{N}}$, then $H=(N:\ideal{h})$. We claim $N$ is $\sigma$--maximal in $\mathcal{N}$. Indeed, if $N\subseteq{L}$ for some $L\in\mathcal{N}$, then $H=(N:\ideal{h})\subseteq(L:\ideal{h})$, by the maximality of $H$ we have $(L:\ideal{h})=(N:\ideal{h})$, hence $L\ideal{h}\subseteq{N}$.
\par
(c) $\Rightarrow$ (a). %
Let $N_1\subseteq{N_2}\subseteq\cdots$ be an increasing chain of submodules of $M$. We consider the family $\mathcal{N}=\{N_s\mid\;s\in\mathbb{N}\setminus\{0\}\}$. By the hypothesis $\mathcal{N}$ has $\sigma$--maximal elements. If $N_m\in\mathcal{N}$ is $\sigma$--maximal, there exists $\ideal{h}\in\mathcal{L}(\sigma)$ such that $N_s\ideal{h}\subseteq{N_m}$ for every $s\geq{m}$; hence $M$ is totally $\sigma$--noetherian.
\end{proof}

Observe that if $M$ is an $A$--module, for any submodule $N\subseteq{M}$ we may consider the family $\mathcal{N}=\{N\}$, and its \textbf{$\sigma$--upper closure}
$$
\overline{\mathcal{N}}=\{H\subseteq{M}\mid\;\textrm{there exists }\ideal{h}\in\mathcal{L}(\sigma)\textrm{ such that }N\subseteq{H}\subseteq(N:\ideal{h})\},
$$
hence for every $H\in\overline{\mathcal{N}}$ we have $N\subseteq{H}\subseteq\Clt{\sigma}{M}{N}$. In addition, we have:

\begin{lemma}
$\overline{\mathcal{N}}$ has only one maximal element if, and only if, $\Clt{\sigma}{M}{N}\in\overline{\mathcal{N}}$.
\end{lemma}
\begin{proof}
Let $H\in\overline{\mathcal{N}}$ be the only maximal element, if there exists $x\in\Clt{\sigma}{M}{N}\setminus{H}$, there are $\ideal{h}\in\mathcal{L}(\sigma)$ such that $H\ideal{h}\subseteq{N}$ and $x\ideal{h}\subseteq{N}$, hence $(H+(x))\ideal{h}\subseteq{N}$, and $H+(x)\in\overline{\mathcal{N}}$, which is a contradiction.
\end{proof}

\begin{lemma}\label{le:20191204b}
Let $M$ be an $A$--module and $T\subseteq{M}$ be a totally $\sigma$--torsion submodule, the following statements are equivalent:
\begin{enumerate}[(a)]\sepa
\item
$M$ is totally $\sigma$--noetherian.
\item
$M/T$ is totally $\sigma$--noetherian.
\end{enumerate}
\end{lemma}
\begin{proof}
(a) $\Rightarrow$ (b). %
Let $N_1/T\subseteq{N_2/T}\subseteq\cdots$ be an increasing chain of submodules of $M/T$, there exist $m\in\mathbb{N}$ and $\ideal{h}\in\mathcal{L}(\sigma)$ such that $N_s\ideal{h}\subseteq{N_m}$, for every $s\geq{m}$. Therefore
\[
\frac{N_s}{T}\ideal{h}=\frac{N_s\ideal{h}+T}{T}\subseteq\frac{N_m}{T}.
\]
\par
(b) $\Rightarrow$ (a). %
Let $N_1\subseteq{N_2}\subseteq\cdots$ be an increasing chain of submodules of $M$, then $(N_1+T)/T\subseteq(N_2+T)/T\subseteq\cdots$ is an increasing chain of submodules of $M/T$, and there exist $m\in\mathbb{N}$, $\ideal{h}\in\mathcal{L}(\sigma)$ such that $\frac{N_s+T}{T}\ideal{h}\subseteq\frac{N_m+T}{T}$, for every $s\geq{m}$. Otherwise, there exists $\ideal{h}'\in\mathcal{L}(\sigma)$ such that $T\ideal{h}'=0$. Therefore,
\[
N_s\ideal{h}\ideal{h}'=(N_s\ideal{h}+T)\ideal{h}'\subseteq(N_m+T)\ideal{h}'=N_m\ideal{h}'\subseteq{N_m}.
\]
\end{proof}

\section{Ring extensions}

Let $\sigma$ be a hereditary torsion theory in $\rMod{A}$, and $f:A\longrightarrow{B}$ be a ring map.

\begin{lemma}
The set $\mathcal{L}(f(\sigma))=\{\ideal{b}\subseteq{B}\mid\;f^{-1}(\ideal{b})\in\mathcal{L}(\sigma)\}$ is a Gabriel filter in $B$, and it defines a hereditary torsion theory in $\rMod{B}$, being
\begin{enumerate}[(1)]\sepa
\item
$\mathcal{T}_{f(\sigma)}=\{M_B\mid\;M_A\in\mathcal{T}_\sigma\}$ and
\item
$\mathcal{F}_{f(\sigma)}=\{M_B\mid\;M_A\in\mathcal{F}_\sigma\}$.
\end{enumerate}
\end{lemma}

We name $f(\sigma)$ the hereditary torsion theory \textbf{induced} by $\sigma$ through the ring map $f$.

\begin{lemma}
Let $\sigma$ be a finite type hereditary torsion theory in $\rMod{A}$, and $f:A\longrightarrow{B}$ be a ring map, the induced hereditary torsion theory $f(\sigma)$ is of finite type.
\end{lemma}
\begin{proof}
Let $\ideal{b}\in\mathcal{L}(f(\sigma))$, there exists $\ideal{a}\in\mathcal{L}(\sigma)$, finitely generated, such that $\ideal{a}\subseteq{f^{-1}(\ideal{b})}$. Therefore, $f(\ideal{a})B\in\mathcal{L}(f(\sigma))$ is finitely generated, and $f(\ideal{a})\subseteq\ideal{b}$, hence $f(\sigma)$ is finitely generated.
\end{proof}

\begin{corollary}[Eakin--Nagata theorem]
Let $\sigma$ be a finite type hereditary torsion theory in $\rMod{A}$, and $f:A\hookrightarrow{B}$ be a ring extension such that $B$ is totally $\sigma$--finitely generated and $B\ideal{p}$ is totally $f(\sigma)$--finitely generated for every prime ideal $\ideal{p}\in\mathcal{K}(\sigma)$, then $A$ is totally $\sigma$--noetherian.
\end{corollary}
\begin{proof}
It is sufficient to prove that $B$ is a totally $\sigma$--noetherian $A$--module because $A$ is submodule of $B$, or equivalently that for every prime ideal $\ideal{p}\in\mathcal{K}(\sigma)$ we have that $\ideal{p}B\subseteq{B}$ is totally $\sigma$--finitely generated. There exists $\ideal{a}\in\mathcal{L}(\sigma)$ such that $\ideal{a}\ideal{p}B\subseteq(p_1,\ldots,p_t)B\subseteq\ideal{p}B$, and there exists $\ideal{a}'\in\mathcal{L}(\sigma)$ such that $\ideal{a}'B\subseteq(b_1,\ldots,b_s)A\subseteq{A}$, hence
\[
\ideal{a}'\ideal{a}\ideal{p}B\subseteq\ideal{a}'(p_1,\ldots,p_t)B\subseteq(p_1,\ldots,p_t)(b_1,\ldots,b_s)A.
\]
\end{proof}

\begin{corollary}[Eakin--Nagata theorem]
Let $\sigma$ be a finite type hereditary torsion theory in $\rMod{A}$, $f:A\hookrightarrow{B}$ be a ring extension such that $B$ is totally $\sigma$--finitely generated and totally $f(\sigma)$--noetherian, then $A$ is totally $\sigma$--noetherian.
\end{corollary}

\begin{proposition}
Let $\sigma$ be a finite type hereditary torsion theory in $\rMod{A}$, and $f:A\hookrightarrow{B}$ be a ring extension such that $B$ is faithfully flat ($\ideal{a}B\cap{A}=\ideal{a}$ for every ideal $\ideal{a}\subseteq{A}$). If $B$ is totally $f(\sigma)$--noetherian, then $A$ is totally $\sigma$--noetherian.
\end{proposition}
\begin{proof}
Let $\ideal{a}\subseteq{A}$, since $\ideal{a}B\subseteq{B}$ is totally $f(\sigma)$--finitely generated, there exists $\ideal{c}\in\mathcal{L}(\sigma)$ such that $\ideal{c}\ideal{a}B\subseteq(b_1,\ldots,b_t)B$, for some $b_1,\ldots,b_t\in\ideal{a}$. Therefore, $\ideal{c}\ideal{a}=\ideal{c}(\ideal{a}B\cap{A})\subseteq(b_1,\ldots,b_t)B\cap{A}=(b_1,\ldots,b_t)A$.
\end{proof}

In order to consider polynomial extensions, we introduce a new kind of finite type hereditary torsion theories. A finite type hereditary torsion theory $\sigma$ in $\rMod{A}$ is \textbf{almost jansian} or \textbf{anti--archimedian} if for every ideal $\ideal{a}\in\mathcal{L}(\sigma)$ we have $\cap_{n=1}^\infty\ideal{a}^n\in\mathcal{L}(\sigma)$.

\medskip
\begin{examples}
\begin{enumerate}[(1)]\sepa
\item
An example of almost jansian hereditary torsion theories are the jansian. A hereditary torsion theory $\sigma$ is jansian whenever $\mathcal{L}(\sigma)$ has a filter basis constituted by an ideal $\ideal{a}$; in this case $\ideal{a}$ must be idempotent. If in addition, $\sigma$ is of finite type then $\ideal{a}$ is finitely generated, hence generated by an idempotent element, say $e\in{A}$, and the localization of $A$ at $\sigma$ is just the ring $eA$.
\item
A multiplicatively closed subset $\Sigma\subseteq{A}$ is anti--archimedean whenever $\cap_{n=1}^\infty{a^nA}\cap{\Sigma}\neq\varnothing$, hence if, and only if, $\sigma_\Sigma$ is almost jansian.
\item
An integral domain $D$ is anti--archimedean if $\cap_{n=1}^\infty{a^nD}\neq0$ for each $a\in{D}$, hence we can rewrite $D$ is anti--archimedean if, and only if, $\sigma_{D\setminus\{0\}}$ is almost jansian.
\item
Let $\ideal{p}\subseteq{A}$ be a prime ideal, the hereditary torsion theory $\sAp$ is of finite type, and it is almost jansian if, and only if, for any ideal $\ideal{a}\subseteq{A}$ if $\ideal{a}\nsubseteq\ideal{p}$, then $\cap_n\ideal{a}^n\nsubseteq\ideal{p}$. In particular, if, and only if, $A/\ideal{p}$ is an anti--archimedean domain; let us call such a $\ideal{p}$ an \textbf{anti--archimedean prime ideal} of $A$.
\item
For every strongly prime ideal $\ideal{p}\subseteq{A}$, see \cite{Jayaram/Oral/Tekir:2013}, the hereditary torsion theory $\sAp$ is almost jansian.
\item
Since the intersection of finitely many finite type hereditary torsion theories is of finite type, if $\{\ideal{p}_1,\ldots,\ideal{p}_t\}$ are anti--archimedean prime ideals of $A$, then $\wedge_{i=1}^t\sigma_{A\setminus\ideal{p}_i}$ is almost jansian.
\end{enumerate}
\end{examples}

\begin{theorem}[Hilbert basis theorem]
Let $\sigma$ be a finite type almost jansian hereditary torsion theory in $\rMod{A}$, and $\sigma'$ the induced hereditary torsion theory in $\rMod{A[X]}$. If $A$ is totally $\sigma$--noetherian, then $A[X]$ is totally $\sigma'$--noetherian.
\end{theorem}
\begin{proof}
Let $\ideal{b}\subseteq{A[X]}$ be an ideal, we define $\ideal{a}=\{\lc(F)\mid\;F\in\ideal{b}\}$, being, for convenience, $\lc(0)=0$. Thus $\ideal{a}\subseteq{A}$ is an ideal, and there exist $\ideal{h}\in\mathcal{L}(\sigma)$, finitely generated, and $a_1,\ldots,a_t\in\ideal{a}$ such that $\ideal{a}\ideal{h}\subseteq(a_1,\ldots,a_t)A\subseteq\ideal{a}$. Let $F_1,\ldots,F_t\in\ideal{b}$ such that $\lc(F_i)=a_i$ for any $i\in\{1,\ldots,t\}$, and $d=\max\{\deg(F_i)\mid\;i\in\{1,\ldots,t\}\}$.
\par
For any $n\in\mathbb{N}\setminus\{0\}$, we define $\mathcal{H}_n=\{F\in\ideal{b}\mid\;\deg(F)<n\}$, thus $\mathcal{H}_n$ is an $A$--module isomorphic to a submodule of the free $A$--module $A^n$, hence it is totally $\sigma$--finitely generated. There exist $\ideal{h}_n\in\mathcal{L}(\sigma)$, finitely generated, and $H_1,\ldots,H_s\in\mathcal{H}_n$ such that $\mathcal{H}_n\ideal{h}_n\subseteq(H_1,\ldots,H_s)A\subseteq\mathcal{H}_n$.
\par
When we take $\mathcal{H}_d$, we may assume that $\ideal{h}_n$ and $\ideal{h}$ are equal. Let us suppose that $\ideal{h}=(h_1,\ldots,h_r)A$.
\par
Let $F\in\ideal{b}$. If $f=\deg(F)<d$, then $F\in\mathcal{H}_d$. If $f=\deg(F)\geq{d}$, we have $\lc(F)\in\ideal{a}$, and $\lc(F)\ideal{h}\subseteq(a_1,\ldots,a_t)A$. For any $j\in\{1,\ldots,r\}$ there exists an $A$--linear combination $\lc(F)h_j=\sum_{i=1}^ta_ic_{i,j}$. Hence there exist natural numbers $e_1,\ldots,e_t$ such that $Fh_j-\sum_{i=1}^tF_iX^{e_i}c_{i,j}=G_j$ is a polynomial in $\ideal{b}$ of degree less than $f$, i.e., $G_j\in\mathcal{H}_f$. Since $Fh_j=\sum_{i=1}^tF_ic_{i,j}+G_j$, then $F\ideal{h}\subseteq(F_1,\ldots,F_t)A+(G_1,\ldots,G_r)A$. We resume this fact saying that for any $0\neq{F}\in\ideal{b}$ of degree $f\geq{d}$, there exists a finite subset $\mathcal{G}\subseteq\mathcal{H}_f$ such that $F\ideal{h}\subseteq(F_1,\ldots,F_t)A+\mathcal{G}A$; if $f<d$, then $\mathcal{G}=\{H_1,\ldots,H_s\}\subseteq\mathcal{H}_d$.
\par
Starting from a polynomial $F\in\ideal{b}$ of degree $f\geq{d}$ there exists a finite subset $\mathcal{G}_1\subseteq\mathcal{H}_f$ such that $f\ideal{h}\subseteq(F_1,\ldots,F_t)A+\mathcal{G}_1A$. For any $G\in\mathcal{G}_1$ there exists $\mathcal{G}'\subseteq\mathcal{H}_{f-1}$ such that $G\ideal{h}\subseteq(F_1,\ldots,F_t)A+\mathcal{G}'A$, hence there exists a finite subset $\mathcal{G}_2\subseteq\mathcal{H}_{f-1}$ such that  $F\ideal{h}^2\subseteq(F_1,\ldots,F_t)A+\mathcal{G}_2A$. And, iterating this process, there exists $k\in\mathbb{N}$ and a finite subset $\mathcal{G}\subseteq\mathcal{H}_d$ such that $F\ideal{h}^k\subseteq(F_1,\ldots,F_t)A+\mathcal{G}A$.
\par
In consequence, $\ideal{b}(\cap_{k\in\mathbb{N}}\ideal{h}^k)\subseteq(F_1,\ldots,F_t)+(\mathcal{G})\subseteq\ideal{b}$, and since $\sigma$ is almost jansian $\cap_k\ideal{h}^k\in\mathcal{L}(\sigma)$, then $\ideal{b}$ is totally $\sigma'$--finitely generated.
\end{proof}

The following consequence also holds.

\begin{corollary}
Let $\sigma$ be a finite type almost jansian hereditary torsion theory in $\rMod{A}$, and $\sigma'$ the induced hereditary torsion theory in $\rMod{A[X_1,\ldots,X_n]}$. If $A$ is totally $\sigma$--noetherian, then $A[X_1,\ldots,X_n]$ is totally $\sigma'$--noetherian.
\end{corollary}

Also we can prove that totally $\sigma$--noetherianness is preserved by localization at multiplicatively subsets.

\begin{proposition}\label{pr:20191116}
Let $A$ be a ring and $\sigma$ be a finite type hereditary torsion theory in $\rMod{A}$. If $\Sigma\subseteq{A}$ is a multiplicatively subset, and we denote by $\sigma'$ the induced hereditary torsion theory in $\rMod{A_\Sigma}$, then $A_\Sigma$ is totally $\sigma'$--noetherian.
\end{proposition}
\begin{proof}
Let $g:A\longrightarrow{A_\Sigma}$ be the canonical ring map. Let $\ideal{p}\subseteq{A_\Sigma}$ be a prime ideal, then $g^{-1}(\ideal{b})\subseteq{A}$ is prime. By the hypothesis there exist $\ideal{h}\in\mathcal{L}(\sigma)$ and $b_1,\ldots,b_t\in{g^{-1}(\ideal{b})}$ such that $g^{-1}(\ideal{b})\ideal{h}\subseteq(b_1,\ldots,b_t)A\subseteq{g^{-1}(\ideal{b})}$. If we localize at $\Sigma$, then
\[
(g^{-1}(\ideal{b})\ideal{h})_\Sigma\subseteq(b_1,\ldots,b_t)A_\Sigma\subseteq{g^{-1}(\ideal{b})}_\Sigma,
\]
i.e.,
\[
\ideal{b}\ideal{h}_\Sigma\subseteq(b_1,\ldots,b_t)A_\Sigma\subseteq\ideal{b}.
\]
Therefore, $\ideal{b}$ is totally $\sigma'$--finitely generated.
\end{proof}

As a consequence we have:

\begin{corollary}
Let $A$ be a ring and $\sigma$ be an almost jansian finite type hereditary torsion theory such that $A$ is totally $\sigma$--noetherian, if $\sigma'$ is the induced hereditary torsion theory in $A[X,X^{-1}]$, then $A[X,X^{-1}]$ is totally $\sigma'$--noetherian.
\end{corollary}
\begin{proof}
Let $A\stackrel{f}{\longrightarrow}A[X]\stackrel{g}{\longrightarrow}A[X,X^{-1}]$. If we consider $\Sigma=\{X^t\mid\;t\in\mathbb{N}\}\subseteq{A[X]}$, then $A[X,X^{-1}]=A[X]_\Sigma$. Since $\sigma$ is almost jansian then $A$ is totally $\sigma$--noetherian, then $A[X]$ is totally $f(\sigma)$--noetherian, and $A[X,X^{-1}]$ is totally $gf(\sigma)$--noetherian, by Proposition~\eqref{pr:20191116}.
\end{proof}

\begin{theorem}[Hilbert basis theorem]
Let $\sigma$ be a finite type almost jansian hereditary torsion theory in $\rMod{A}$, and $\sigma'$ the induced hereditary torsion theory in $\rMod{\serial{A}{X}}$. If $A$ is totally $\sigma$--noetherian, then $\serial{A}{X}$ is totally $\sigma'$--noetherian.
\end{theorem}
\begin{proof}
Let $\ideal{a}\subseteq\serial{A}{X}$ be an ideal. For any $n\in\mathbb{N}$ we define $\ideal{a}_n=\langle\textrm{first coefficient of a }F\in\ideal{a}\cap(X^n)\rangle$.
Hence we obtain an increasing chain $\ideal{a}_0\subseteq\ideal{a}_1\subseteq\cdots$. Since $A$ is totally $\sigma$--noetherian, there exist $m\in\mathbb{N}$ and $\ideal{h}_*\in\mathcal{L}(\sigma)$, finitely generated, such that $\ideal{a}_s\ideal{h}_*\subseteq\ideal{a}_m$ for any $s\geq{m}$. Otherwise, for any $\ideal{a}_n$ there exists a finitely generated ideal $\ideal{a}'_n\subseteq\ideal{a}_n$, and $\ideal{h}_n\in\mathcal{L}(\sigma)$ such that $\ideal{a}_n\ideal{h}_n\subseteq\ideal{a}'_n$. If we take $\ideal{h}=\ideal{h}_*\ideal{h}_1\cdots\ideal{h}_m$, then
\[
\begin{array}{ll}
\ideal{a}_n\ideal{h}\subseteq\ideal{a}'_n,\textrm{ for any }n\leq{m}\\
\ideal{a}_n\ideal{h}\subseteq\ideal{a}\ideal{h}_*\ideal{h}_m\subseteq\ideal{a}_m\ideal{h}_m\subseteq\ideal{a}'_m,
\textrm{ for any }n>m.
\end{array}
\]
For any $n=0,1,\ldots,m$, if $\ideal{a}'_n=\langle{a_{n,1},\ldots,a_{n,t_n}}\rangle$, we take $F_{i,j}\in\ideal{a}\cap(X^n)$ such that the first coefficient of $F_{i,j}$ is $a_{i,j}$.
\par
For any $F\in\ideal{a}$ we have: $F\ideal{h}\in(\ideal{a}'_0+\ideal{a}'_1X+\cdots)\cap\ideal{a}\cap(X^n)$. If $F=(b_0+b_1X+\cdots)X^n\in\ideal{a}\cap(X^n)$, then $F\ideal{h}\in(\ideal{a}'_nX^n+\cdots)\cap\ideal{a}\cap(X^n)$, and, for every $h\in\ideal{h}$ there exists a linear combination $b_0h=\sum_{j=1}^{t_n}a_{n,j}c_{n,j}$, hence $Fh-\sum_{j=1}^{t_n}F_{n,j}c_{n,j}\in\ideal{a}\cap(X^{n+1})$. In the same way, for every $h'\in\ideal{h}$ there
exists a linear combination
$Fhh'-\sum_{j=1}^{t_n}F_{n,j}c_{n,j}h'-\sum_{j=1}^{t_{n+1}}F_{n+1,j}c_{n+1,j}\in\ideal{a}\cap(X^{n+2})$.
\par
This reduces the problem to study what happens with the elements in $\ideal{a}\cap(X^m)$. As we saw before, if $F\in\ideal{a}\cap(X^m)$, for any $h_1\in\ideal{h}$ there exists a linear combination $Fh_1-\sum_{j=1}^{t_m}F_{m,j}c_{1,m,j}\in\ideal{a}\cap(X^{m+1})$. Since $\ideal{a}_{m+1}\ideal{h}\subseteq\ideal{a}'_m$, for any $h_2\in\ideal{h}$ there exists a linear combination
$Fh_1h_2-\sum_{j=1}^{t_m}F_{m,j}c_{1,m,j}h_2-\sum_{j=1}^{t_m}F_{m,j}c_{2,m,j}\in\ideal{a}\cap(X^{m+2})$; therefore,
$Fh_1h_2-\sum_{j=1}^{t_m}F_{m,j}(c_{1,m,j}h_2+c_{2,m,j})\in\ideal{a}\cap(X^{m+2})$. In general,
$$
Fh_1h_2\cdots{h_s}-\sum_{j=1}^{t_m}F_{m,j}(c_{1,m,j}h_2\cdots{h_s}+c_{2,m,j}h_3\cdots{h_s}+\cdots+c_{s,m,j})
\in\ideal{a}\cap(X^{m+s}).
$$
\par
Since $\sigma$ is almost jansian, i.e., $\cap_n\ideal{h}^n\in\mathcal{L}(\sigma)$, if $\ideal{h}_0=\cap_n\ideal{h}^n$, then $F\ideal{h}_0\subseteq\langle{F_{m,1},\ldots,F_{m,t_m}}\rangle$.
\end{proof}

\begin{corollary}
Let $\sigma$ be a finite type almost jansian hereditary torsion theory in $\rMod{A}$, and $\sigma'$ the induced hereditary torsion theory in $\rMod{\serial{A}{X_1,\ldots,X_n}}$. If $A$ is totally $\sigma$--noetherian, then $\serial{A}{X_1,\ldots,X_n}$ is totally $\sigma'$--noetherian.
\end{corollary}

\section{Study through prime ideals}

Let $\ideal{p}\subseteq{A}$ be a prime ideal, we consider $\sAp$, the hereditary torsion theory cogenerated by $A/\ideal{p}$, or equivalently, the hereditary torsion theory generated by the multiplicatively subset $A\setminus\ideal{p}$. For every torsion theory $\sigma$ we consider the following sets of ideals:
\begin{enumerate}[(1)]\sepa
\item
$\mathcal{L}(\sigma)$, the Gabriel filter of $\sigma$.
\item
$\mathcal{Z}(\sigma)=\mathcal{L}(\sigma)\cap\Spec(A)$. In particular, if $\ideal{p}\subseteq\ideal{q}$ are prime ideals and $\ideal{p}\in\mathcal{Z}(\sigma)$, then $\ideal{q}\in\mathcal{Z}(\sigma)$.
\item
$C(A,\sigma)=\{\ideal{a}\mid\;A/\ideal{a}\in\mathcal{F}_\sigma\}$.
\item
$\mathcal{K}(\sigma)=C(A,\sigma)\cap\Spec(A)$; it is the complement of $\mathcal{Z}(\sigma)$ in $\Spec(A)$. In particular, if $\ideal{p}\subseteq\ideal{q}$ are prime ideals and $\ideal{q}\in\mathcal{Z}(\sigma)$, then $\ideal{p}\in\mathcal{Z}(\sigma)$.
\item
$\mathcal{C}(\sigma)=\Max\mathcal{K}(\sigma)$.
\end{enumerate}
If $\sigma$ is of finite type, then $\sigma=\wedge\{\sAp\mid\;\ideal{p}\in\mathcal{K}(\sigma)\}$. Otherwise, $\sigma=\wedge\{\sAp\mid\;\ideal{p}\in\mathcal{C}(\sigma)\}$ whenever $A$ is $\sigma$--noetherian, because $\sigma_{A\setminus\ideal{q}}\leq\sAp$ if $\ideal{p}\subseteq\ideal{q}$, for any prime ideals $\ideal{p},\ideal{q}$.

An $A$--module $M$ is \textbf{totally $\ideal{p}$--noetherian} whenever $M$ is totally $\sAp$--noetherian.

\begin{lemma}
Let $A$ be a local (non necessarily noetherian) ring with maximal ideal $\ideal{m}$, and $M$ be an $A$--module. The following statements are equivalent:
\begin{enumerate}[(a)]\sepa
\item
$M$ is noetherian.
\item
$M$ is totally $\sAm$--noetherian.
\end{enumerate}
\end{lemma}
\begin{proof}
It is immediate because every element in $A\setminus\ideal{m}$ is invertible.
\end{proof}

\begin{proposition}\label{pr:20191030}
Let $\sigma$ a finite type hereditary torsion theory in $A$, and $M$ be an $A$--module. The following statements are equivalent:
\begin{enumerate}[(a)]\sepa
\item
$M$ is totally $\sigma$--noetherian.
\item
$M$ is totally $\sAp$--noetherian for every $\ideal{p}\in\mathcal{C}(\sigma)=\Max\mathcal{K}(\sigma)$.
\end{enumerate}
\end{proposition}
\begin{proof}
(a) $\Rightarrow$ (b). %
It is immediate because $\sigma\leq\sAp$.
\par
(b) $\Rightarrow$ (a). %
Let $N\subseteq{M}$, for every $\ideal{m}\in\mathcal{C}(\sigma)$ there exists $s_\ideal{m}\in{A\setminus\ideal{m}}$, and $H_\ideal{m}\subseteq{N}$, finitely generated, such that $Ns_\ideal{m}\subseteq{H_\ideal{m}}\subseteq{N}$. Since $\ideal{b}=\sum_\ideal{m}s_\ideal{m}A$ belongs to $\mathcal{L}(\sigma)=\cap_\ideal{m}\mathcal{L}(\sAm)$, there exists $\ideal{c}\in\mathcal{L}(\sigma)$, finitely generated, such that $\ideal{c}\subseteq\ideal{b}$. In consequence, there are finitely many elements $s_{\ideal{m}_1},\ldots,s_{\ideal{m}_t}$ such that $\ideal{c}\subseteq\sum_{i=1}^ts_{\ideal{m}_i}A\subseteq\ideal{b}$, and we have $N\ideal{c}\subseteq\sum_{i=1}^tH_{\ideal{m}_i}\subseteq{N}$. Hence, $N$ is totally $\sigma$--finitely generated, and $M$ is totally $\sigma$--noetherian.
\end{proof}

When we take the hereditary torsion theory $\sigma=0$, i.e, when $\mathcal{L}(\sigma)=\{A\}$, we have that $M$ is noetherian if, and only if, $M$ is totally $\sAm$--noetherian for every maximal ideal $\ideal{m}\in\Supp(M)$, which is \cite[Proposition~12]{Anderson/Dumitrescu:2002}.

\section{Principal ideal rings}

Let $\sigma$ be a hereditary torsion theory in $\rMod{A}$, and $\ideal{a}\subseteq{A}$ an ideal, then we have:
\begin{enumerate}[(1)]\sepa
\item
$\ideal{a}$ is \textbf{$\sigma$--principal} if there exists $a\in\ideal{a}$ such that $\Clt{\sigma}{A}{aA}=\Clt{\sigma}{A}{\ideal{a}}$.
\item
$A$ is a \textbf{$\sigma$--principal ideal ring}, $\sigma$-PIR, whenever every ideal is $\sigma$--principal.
\item
$A$ is a \textbf{totally $\sigma$--principal ideal ring}, totally $\sigma$--PIR, whenever every ideal is totally $\sigma$--principal.
\end{enumerate}

\begin{proposition}[Kaplansky's Theorem]
Let $\sigma$ be a finite type hereditary torsion theory in $\rMod{A}$, the following statements are equivalent:
\begin{enumerate}[(a)]\sepa
\item
$A$ is totally $\sigma$--PIR.
\item
Every prime ideal $\ideal{p}\in\mathcal{K}(\sigma)$ is totally $\sigma$--principal.
\end{enumerate}
\end{proposition}

See also \cite[Proposition~16]{Anderson/Dumitrescu:2002}.

\begin{proof}
(a) $\Rightarrow$ (b). %
Since for every $\ideal{p}\in\mathcal{K}(\sigma)$, the result holds.
\par
(b) $\Rightarrow$ (a). %
Let $\Gamma=\{\ideal{a}\subseteq{A}\mid\;\ideal{a}\textrm{ is not totally $\sigma$--principal}\}$. If $\Gamma\neq\varnothing$, since every chain in $\Gamma$ has a upper bound in $\Gamma$, there exist maximal elements in $\Gamma$. If $\{\ideal{a}_i\}_i\subseteq\Gamma$ is a chain, we take $\ideal{a}=\cup_i\ideal{a}_i$. If $\ideal{a}$ is not totally $\sigma$--principal, there are $\ideal{b}\in\mathcal{L}(\sigma)$, finitely generated, and $a\in\ideal{a}$ such that $\ideal{a}\ideal{b}\subseteq{aA}\subseteq\ideal{a}$, and there exists an index $i$ such that $a\in\ideal{a}_i$, hence $\ideal{a}_i\ideal{b}\subseteq\ideal{a}\ideal{b}\subseteq{aA}\subseteq\ideal{a}_i$, which is a contradiction.
\par
Let $\ideal{a}\in\Gamma$ maximal; if $\Clt{\sigma}{A}{\ideal{a}}$ is totally $\sigma$--principal, then $\ideal{a}$ is totally $\sigma$--principal, hence $\ideal{a}=\Clt{\sigma}{A}{\ideal{a}}$ is $\sigma$--closed.
\par
Let us show that $\ideal{a}$ is prime. Let $a,b\in{A\setminus\ideal{a}}$ such that $ab\in\ideal{a}$. Since $\ideal{a}+aA$ is totally $\sigma$--principal, there exists $x\in\ideal{a}+aA$ and $\ideal{b}\in\mathcal{L}(\sigma)$, finitely generated, such that $(\ideal{a}+aA)\ideal{b}\subseteq{xA}\subseteq\ideal{a}+aA$. If $x\in{A}$, then $\ideal{b}\subseteq(\ideal{a}:a)\notin\mathcal{L}(\sigma)$, which is a contradiction. If $x\notin{A}$, then $\ideal{b}\subseteq(\ideal{a}:x)\notin\mathcal{L}(\sigma)$, which is a contradiction.
In consequence, we can find a prime ideal $\ideal{a}\in\mathcal{K}(\sigma)$ which is not totally $\sigma$--principal which is a contradiction.
\end{proof}

\begin{corollary}
Let $\sigma$ be a finite type hereditary torsion theory, the following statements are equivalent:
\begin{enumerate}[(a)]\sepa
\item
$A$ is totally $\sigma$--PIR.
\item
$A$ is $\sigma$-PIR and every $\ideal{p}$ is totally $\sigma$--finitely generated.
\end{enumerate}
In particular, $A$ is totally $\sigma$-PIR if, and only if, $A$ is $\sigma$--PIR and totally $\sigma$--noetherian.
\end{corollary}
\begin{proof}
(a) $\Rightarrow$ (b). %
It is evident.
\par
(b) $\Rightarrow$ (a). %
Let $\ideal{p}\in\mathcal{K}(\sigma)$, by the hypothesis, there exist $\ideal{h}_1\in\mathcal{L}(\sigma)$, finitely generated, and $p_1,\ldots,p_t\in\ideal{p}$ such that $\ideal{p}\ideal{h}_1\subseteq(p_1,\ldots,p_t)\subseteq\ideal{p}$. Otherwise, there exists $p\in(p_1,\ldots,p_t)$ and $\ideal{h}_2\in\mathcal{L}(\sigma)$, finitely generated, such that $(p_1,\ldots,p_t)\ideal{h}_2\subseteq{pA}\subseteq(p_1,\ldots,p_t)$. Observe that $\Clt{\sigma}{A}{pA}=\ideal{p}=\Clt{\sigma}{A}{p_1,\ldots,p_t}$. Therefore we have
\[
\ideal{p}\ideal{h}_1\ideal{h}_2
\subseteq(p_1,\ldots,p_t)\ideal{h}_2
\subseteq{pA}
\subseteq\ideal{p},
\]
and $\ideal{p}$ is totally $\sigma$--finitely generated.
\end{proof}

\providecommand{\bysame}{\leavevmode\hbox to3em{\hrulefill}\thinspace}
\providecommand{\MR}{\relax\ifhmode\unskip\space\fi MR }
\providecommand{\MRhref}[2]{%
  \href{http://www.ams.org/mathscinet-getitem?mr=#1}{#2}
}
\providecommand{\href}[2]{#2}

\end{document}